\documentclass[a4paper,12pt]{amsart}

\usepackage{amsmath, amsfonts, amsthm, bbm, microtype}
\usepackage[math]{anttor}
\usepackage[T1]{fontenc}

\newtheorem*{theorem}{Theorem}
\newtheorem*{corollary}{Corollary}
\newtheorem{lemma}{Lemma}
\newtheorem{proposition}{Proposition}
\theoremstyle{definition}
\newtheorem*{example}{Example}

\DeclareMathOperator{\var}{var}

\begin{document}

\author{Tomas Persson}

\title{A strong Borel--Cantelli lemma for recurrence}

\date{\today}



\address{T.~Persson, Centre for Mathematical Sciences, Lund
  University, Box~118, 221~00~Lund, Sweden}

\email{tomasp@maths.lth.se}

\thanks{I thank Philipp Kunde for commenting on the manuscript,
  and I thank the referee for suggesting the inclusion of
  Section~\ref{sec:slowmixing}. I also thank Alejandro Rodriguez
  Sponheimer for pointing out several small errors, that have
  been corrected in the present version.}

\begin{abstract}
  Consider a dynamical systems $([0,1], T, \mu)$ which is
  exponentially mixing for $L^1$ against bounded variation. Given
  a non-summable sequence $(m_k)$ of non-negative numbers, one
  may define $r_k (x)$ such that $\mu (B(x, r_k(x)) = m_k$. It is
  proved that for almost all $x$, the number of $k \leq n$ such
  that $T^k (x) \in B_k (x)$ is approximately equal to
  $m_1 + \ldots + m_n$. This is a sort of strong Borel--Cantelli
  lemma for recurrence.

  A consequence is that
  \[
    \lim_{r \to 0} \frac{\log \tau_{B(x,r)} (x)}{- \log \mu (B (x,r))}
    = 1
  \]
  for almost every $x$, where $\tau$ is the return time.
\end{abstract}

\maketitle

\section{Recurrence}

A classical subject of study in the theory of dynamical systems
is the concept of recurrence, which goes back to the well known
Poincar\'{e} recurrence theorem, proved by Carath\'{e}odory
\cite{Caratheodory}.

Boshernitzan proved that if $(X, T, \mu)$ is a measure preserving
system of a metric space $X$ which is $\sigma$-finite with respect
to the $\alpha$ dimensional Hausdorff measure, then
\[
  \liminf_{k \to \infty} k^\frac{1}{\alpha} d(T^k (x), x) <
  \infty
\]
for $\mu$ almost every $x$. In other words, for almost every $x$,
there is a $c > 0$ such that
$T^k (x) \in B(x, c k^{-\frac{1}{\alpha}})$ for infinitely many
$k$. Yet another way to say this is that for almost all $x$,
there is a constant $c > 0$ such that
\[
  \sum_{k = 1}^\infty \mathbf{1}_{B (x, c k^{-\frac{1}{\alpha}})}
  (T^k (x)) = \infty,
\]
where $\mathbf{1}_{E}$ denotes the indicator function of the set
$E$.

Recent improvements of Boshernitzan's result for particular
classes of dynamical systems can be found in papers by Pawelec
\cite{Pawelec}; Chang, Wu and Wu \cite{Chang}; Baker and Farmer
\cite{Baker}; Hussain, Li, Simmons and Wang \cite{Husseinetal};
and by Kirsebom, Kunde and Persson \cite{KirsebomKundePersson}.

In this paper we shall refine the investigation and consider the
typical growth speed of
\[
  n \mapsto \sum_{k = 1}^n \mathbf{1}_{B (x, r_k)} (T^k
  (x)),
\]
for certain sequences $r_k$. That is, we shall count the number
of close returns before time $n$ of a point $x$ to itself.

In the setting of hitting a shrinking target $B(y_k, r_k)$, one
can sometimes prove that for almost all $x$ there are infinitely
many $k$ such that $T^k (x) \in B (y_k, r_k)$ provided
$\sum_{k = 1}^\infty \mu (B (y_k, r_k)) = \infty$. Such results
are called dynamical Borel--Cantelli lemmata. It is sometimes
possible to prove the stronger result that
\[
  \sum_{k = 1}^n \mathbf{1}_{B (y_k, r_k)} (T^k (x)) \sim \sum_{k
    = 1}^n \mu (B (y_k, r_k))
\]
for almost all $x$, see for instance Kim \cite{Kim}. Such results
are called strong dynamical Borel--Cantelli lemmata. Note that
here the centre $y_k$ of the target $B(y_k, r_k)$ may depend on
$k$, but does not depend on $x$ as in the results for recurrence
mentioned above.

The result of this paper is a sort of strong Borel--Cantelli
lemma for recurrence. In a certain sense, we shall obtain that
\[
  \sum_{k = 1}^n \mathbf{1}_{B (x, r_k)} (T^k (x)) \sim \sum_{k =
    1}^n \mu (B (x, r_k))
\]
holds for almost all $x$, under some rather mild assumptions.

For a measure preserving dynamical system $([0,1], T, \mu)$, we say
that correlations decay exponentially for $L^1$ against $BV$, if there
are constants $c, \tau > 0$ such that
\[
  \biggl| \int f \circ T^n g \, \mathrm{d} \mu - \int f \,
  \mathrm{d} \mu \int g \, \mathrm{d} \mu \biggr| \leq c e^{-
    \tau n} \lVert f \rVert_1 \lVert g \rVert_{BV},
\]
where $\lVert f \rVert_1 = \int |f| \, \mathrm{d} \mu$ and $\lVert g
\rVert_{BV}$ is the bounded variation norm, defined by
\[
\lVert g \rVert_{BV} = \sup |g| + \var g,
\]
where $\var g$ is the total variation of $g$.
It is well known that correlations decay exponentially for $L^1$
against $BV$, for instance if $T$ is a piecewise uniformly expanding
map of $[0,1]$ and $\mu$ is a Gibbs measure \cite{LasotaYorke,
  Rychlik, Liveranietal}. In that case, it is often the case that
there are constants $c, s > 0$ such that $\mu (B(x, r)) \leq c r^s$
always hold.

The result of this paper is the following theorem, which can be
thought of as a strong Borel--Cantelli lemma for recurrence.

\begin{theorem}
  Suppose that $T \colon [0,1] \to [0,1]$ has an invariant
  measure $\mu$ such that correlations for $L^1$ against $BV$
  decay exponentially, and that there are constants $c, s > 0$
  such that
  \[
    \mu (B(x,r)) \leq c r^s
  \]
  holds for all points $x$ and all $r > 0$.
  
  Let $(m_k)$ be a sequence such that
  \[
    m_k \geq \frac{(\log k)^{4 + \varepsilon}}{k}
  \]
  for some $\varepsilon > 0$, and such that
  \[
    \lim_{\rho \to 1^+} \limsup_{k \to \infty}
    \frac{m_k}{m_{[\rho k]}} = 1.
  \]

  Let $B_k (x)$ be the ball with centre $x$ and
  $\mu (B_k (x)) = m_k$. Then $\mu$ almost every $x$ satisfies
  \[
    \lim_{n \to \infty} \frac{\sum_{k = 1}^n \mathbf{1}_{B_k (x)}
      (T^k (x))}{\sum_{k = 1}^n \mu (B_k (x))} = 1.
  \]
\end{theorem}

The last section of this paper contains two examples of systems
for which results of this type do not hold.

\section{Overview of the proof}

Let $E_k = \{\, x : T^k (x) \in B_k (x) \,\}$.  Then
$T^k (x) \in B_k (x)$ if and only if $x \in E_k$. Hence, the goal
is to estimate the sum
\[
  \sum_{k = 1}^n \mathbf{1}_{E_k} (x).
\]
It turns out that it is difficult to get good enough estimates to
handle this sum directly. The idea is therefore to instead consider
the sum
\[
  S_n (x) = \sum_{k = 1}^n \frac{\mathbf{1}_{E_k} (x)}{\mu
    (E_k)},
\]
which has the useful property that
\[
  \int S_n \, \mathrm{d} \mu = n.
\]

Using estimates on correlations, one obtains through
Lemma~\ref{lem:Sprindzuk} from Section~\ref{sec:lemmata}, the
asymptotic estimate
\begin{equation}
  \label{eq:asympt}
  S_n (x) = \sum_{k = 1}^n \frac{\mathbf{1}_{B_k (x)} (T^k (x))}{\mu
    (E_k)} \sim \sum_{k = 1}^n \frac{\mathbf{1}_{B_k (x)} (T^k (x))}{\mu
    (B_k (x))} \sim n
\end{equation}
for almost all $x$. Under some assumptions on the sequence
$m_k = \mu (B_k(x))$, Lemma~\ref{lem:abel} of the same section,
implies that a consequence of \eqref{eq:asympt} is that
\[
  \frac{ \sum_{k = 1}^n \mathbf{1}_{B_k (x)} (T^k (x))}{\sum_{k = 1}^n
    \mu (B_k(x)) } \to 1
\]
as $n \to \infty$.

The above mentioned lemmata are stated and proved in
Section~\ref{sec:lemmata}. The proof of the Theorem is in
Section~\ref{sec:proof}.

\section{An application to return times}

We give an application of the Theorem in this paper.  As
mentioned in the introduction, a strong Borel--Cantelli lemma is
a statement of the type that
\[
  \lim_{n \to \infty} \frac{\sum_{k=1}^n \mathbf{1}_{B(y, r_k)}
    (T^k (x))}{\sum_{k=1}^n \mu (B (y, r_k))} = 1
\]
holds for almost every $x$ when
$\sum_{k=1}^\infty \mu (B (y, r_k)) = \infty$. Galatolo and Kim
\cite{GalatoloKim} used strong Borel--Cantelli lemmata, to conclude
that
\[
  \lim_{n \to \infty} \frac{\log \tau_{B (y, r_n)} (x)}{- \log
    \mu (B(y, r_n))} = 1
\]
holds for almost every $x$, where $\tau_B$ is the hitting time to $B$,
defined by
\[
  \tau_B (x) = \min \{\, k > 0 : T^k (x) \in B \,\}.
\]
When $x \in B$, the time $\tau_B (x)$ is usually called the return
time of $x$ to $B$.

Adapting the proof of Galatolo and Kim, one obtains the following
result on return times as a corollary to the Theorem.

\begin{corollary}
  Under the assumptions of the Theorem, we have
  \[
    \lim_{r \to 0} \frac{\log \tau_{B(x, r)} (x)}{- \log \mu
      (B(x, r))} = 1
  \]
  for $\mu$ almost every $x$.
\end{corollary}

The proof of this corollary is in Section~\ref{sec:corollary}.

If we let
\[
  \underline{d}_\mu (x) = \liminf_{r \to 0} \frac{\log \mu (B (x,
    r))}{\log r}
\]
and
\[
  \overline{d}_\mu (x) = \limsup_{r \to 0} \frac{\log \mu (B (x,
    r))}{\log r}
\]
be the lower and upper pointwise dimensions of the measure $\mu$
at the point $x$, then it follows from the Corollary that
\[
  \underline{d}_\mu (x) = \liminf_{r \to 0} \frac{\log \tau_{B
      (x, r)} (x)}{- \log r} \leq \limsup_{r \to 0} \frac{\log
    \tau_{B (x, r)} (x)}{- \log r} = \overline{d}_\mu (x)
\]
holds for almost every $x$. In many cases it is known that
$\underline{d}_\mu (x) = \overline{d}_\mu (x)$ holds for almost
all $x$ so that in fact we have
\begin{equation}
  \label{eq:BS}
  \underline{d}_\mu (x) = \liminf_{r \to 0} \frac{\log \tau_{B
      (x, r)} (x)}{- \log r} = \limsup_{r \to 0} \frac{\log
    \tau_{B (x, r)} (x)}{- \log r} = \overline{d}_\mu (x)
\end{equation}
for almost all $x$.  The equalities in \eqref{eq:BS} have been
proved by Barreira and Saussol \cite{BarreiraSaussol} for
$C^{1+\alpha}$-diffeomorphisms when $\mu$ is an equilibrium
measure.

\section{Lemmata} \label{sec:lemmata}

The following lemma is a variation on a lemma by Sprindžuk
\cite[Lemma~10, page~45]{Sprindzuk}. The assumptions are a little
bit different, and the proof is similar, but easier.

\begin{lemma}
  \label{lem:Sprindzuk}
  Let $(f_k)$ be a sequence of integrable functions and
  $(\phi_k)$ a sequence of numbers such that $\phi_k \geq 1$ for
  all $k$.
  
  If for all $m < n$, we have
  \begin{equation}
    \label{eq:sumassumption}
    \int \biggl( \sum_{m < k \leq n} \biggl( f_k - \int f_k \,
    \mathrm{d} \mu \biggr) \biggr)^2 \, \mathrm{d} \mu \leq
    \sum_{m < k \leq n} \phi_k,
  \end{equation}
  then for every $\varepsilon > 0$, we have
  \[
    \sum_{k = 1}^n f_k (x) = \sum_{k = 1}^n \int f_k \,
    \mathrm{d} \mu + O \Bigl( (\log \Phi (n + 1))^{\frac{3}{2} +
      \varepsilon} \Phi (n + 1)^\frac{1}{2} \Bigr),
  \]
  for $\mu$ almost every $x$, where
  \[
    \Phi (n) = \sum_{k = 1}^n \phi_k.
  \]
\end{lemma}

Before proving Lemma~\ref{lem:Sprindzuk}, we state the following
lemma.

\begin{lemma} \label{lem:abel}
  Let $(a_k)$ be a decreasing sequence of non-negative numbers
  such that
  \[
    \sum_{k = 1}^\infty a_k = \infty
  \]
  and
  \begin{equation}
    \label{eq:akassumption}
    \lim_{\rho \to 1^+} \limsup_{k \to \infty}
    \frac{a_k}{a_{[\rho k]}} = 1.
  \end{equation}
  Suppose that $(x_k)$ is a sequence of non-negative numbers such
  that
  \[
    \lim_{n \to \infty} \frac{1}{n} \sum_{k = 1}^n
    \frac{x_k}{a_k} = 1.
  \]
  Then
  \[
    \lim_{n \to \infty} \frac{\sum_{k = 1}^n x_k}{\sum_{k = 1}^n
      a_k} = 1.
  \]
\end{lemma}

We shall now prove the two lemmata, and start with
Lemma~\ref{lem:Sprindzuk}.

\begin{proof}[Proof of Lemma~\ref{lem:Sprindzuk}]
  If $I = (m, n]$ and $m$ and $n$ are integers, then we write
  \[
    \Phi (I) = \sum_{k \in I} \phi_k, \qquad \text{and} \qquad F
    (I, x) = \sum_{k \in I} f_k (x).
  \]
  Let $\Phi (n) = \Phi((0,n])$ and $F (n, x) = F ((0,n], x)$.

  When $u$ is a natural number, we let $n_u$ be the largest $n$
  for which $\Phi (n) < u$. Hence
  $\Phi (n_u) < u \leq \Phi (n_u + 1)$. If $u < v$, the interval
  $I = (u, v]$ is non-empty, but it can happen that $(n_u, n_v]$
  is empty.

  We define $\sigma ((u, v]) = (n_u, n_v]$. It follows that if
  $I = I_1 \cup I_2$, then
  $\sigma (I) = \sigma (I_1) \cup \sigma (I_2)$. Moreover, if the
  union $I_1 \cup I_2$ is disjoint, then so is
  $\sigma (I_1) \cup \sigma (I_2)$.

  For integers $0 \leq s \leq r$, we let $J_{r,s}$ be the set of
  intervals of the form $(i 2^s, (i+1) 2^s]$,
  $0 \leq i < 2^{r-s}$. Then
  \[
    \bigcup_{I \in J_{r,s}} \sigma(I) = \sigma ((0, 2^r]) = (0,
    n_{2^r}],
  \]
  is a disjoint union. Therefore,
  \[
    \sum_{I \in J_{r,s}} \Phi (\sigma(I)) = \Phi \biggl(
    \bigcup_{I \in J_{r,s}} \sigma(I) \biggr) = \Phi (n_{2^r}) <
    2^r.
  \]

  With $J_r = \bigcup_{s = 0}^r J_{r,s}$ we then have
  \[
    \sum_{I \in J_r} \Phi (\sigma(I)) = \sum_{s = 0}^r \sum_{I
      \in J_{r,s}} \Phi (\sigma(I)) < (r + 1) 2^r.
  \]
  Put
  \[
    g (r, x) = \sum_{I \in J_r} \biggl( F (\sigma(I), x) - \int F
    (\sigma(I), x) \, \mathrm{d} \mu (x) \biggr)^2.
  \]
  Then by \eqref{eq:sumassumption}, we have that
  \[
    \int g (r, x) \, \mathrm{d} \mu (x) \leq \sum_{I \in J_r}
    \Phi (\sigma(I)) < (r + 1) 2^r.
  \]
  It follows that
  \[
    \mu \{\, x : g(r, x) \geq (r + 1) r^{1 + 2 \varepsilon} 2^r
    \,\} < r^{-1 - 2 \varepsilon}
  \]
  where $\varepsilon > 0$. By the Borel--Cantelli lemma we then
  get that for almost all $x$,
  \[
    g (r, x) < (r + 1) r^{1 + 2 \varepsilon} 2^r
  \]
  holds for all large $r$.

  An interval $(0, v]$ is a union of at most $r = [\log_2 v] + 1$
  intervals from $J_r$. We let $J_r (v)$ denote the set of those
  intervals. Then $(0, n_v]$ is a disjoint union of the intervals
  $\sigma (I)$, $I \in J_r (v)$. We have
  \[
    F (n_v, x) - \int F (n_v, x) \, \mathrm{d} \mu (x) = \sum_{I
      \in J_r (v)} \biggl( F (\sigma(I), x) - \int F (\sigma(I),
    x) \, \mathrm{d} \mu (x) \biggr).
  \]
  and by the Cauchy--Bunyakovsky--Schwarz inequality, we have
  that
  \begin{align*}
    \biggl( F (n_v, x)
    &- \int F (n_v, x) \, \mathrm{d} \mu (x) \biggr)^2 \\
    & \leq \biggl( \sum_{I \in J_r (v)} 1^2 \biggr) \sum_{I \in
      J_r (v)} \biggl( F (\sigma(I), x) - \int F (\sigma(I), x)
      \, \mathrm{d} \mu (x) \biggr)^2 \\
    & \leq r \sum_{I \in J_r (v)} \biggl( F (\sigma(I), x) - \int
      F (\sigma(I), x) \, \mathrm{d} \mu (x) \biggr)^2 \\
    & \leq r g(r, x),
  \end{align*}
  since the sum has at most $r$ terms. Hence, for almost every
  $x$, we have
  \begin{align*}
    \sum_{k = 1}^{n_v} f_k (x)
    &= F (n_v, x) \leq \int F (n_v, x) \, \mathrm{d} \mu (x) +
      \sqrt{r g(r, x)} \\
    &= \int F (n_v, x) \, \mathrm{d} \mu (x) + O
      \Bigl( r^{\frac{3}{2} + \varepsilon}
      2^{\frac{r}{2}} \Bigr) \\    
    &= \sum_{k = 1}^{n_v} \int f_k (x) \, \mathrm{d}
      \mu (x) + O \Bigl( (\log v)^{\frac{3}{2} + \varepsilon}
      v^{\frac{1}{2}} \Bigr) \\
    &= \sum_{k = 1}^{n_v} \int f_k (x) \, \mathrm{d}
      \mu (x) + O \Bigl( (\log \Phi(n_{v}+1))^{\frac{3}{2} +
      \varepsilon} \Phi(n_{v}+1)^{\frac{1}{2}} \Bigr).
  \end{align*}

  Consider an arbitrary positive integer $n$. Since $\phi_k \geq
  1$, there is a $v$ such that $n = n_v$. This finishes the proof.
\end{proof}

\begin{proof}[Proof of Lemma~\ref{lem:abel}]
  Put
  \[
    \sigma_n = \frac{1}{n} \sum_{k = 1}^n \frac{x_k}{a_k}
  \]
  and let $\varepsilon > 0$. Take $n_1$ such that
  $|\sigma_n - 1| < \varepsilon$ for all $n \geq n_1$.

  Let $\rho > 1$ and put $n_k = [\rho^{k - 1} n_1]$. For every
  $k$, the number
  \[
    \Delta_k := n_{k+1} \sigma_{n_{k+1}} - n_k \sigma_{n_k} =
    \sum_{l = n_k + 1}^{n_{k+1}} \frac{x_l}{a_l}
  \]
  satisfies
  \begin{align*}
    \Delta_k & \leq n_{k+1} (1 + \varepsilon) - n_k (1 -
               \varepsilon) \\
             & = (n_{k+1} - n_k) (1 + \varepsilon) + 2
               \varepsilon n_k               
  \end{align*}
  and
  \[
    \Delta_k \geq (n_{k+1} - n_k) (1 - \varepsilon) - 2
    \varepsilon n_k.
  \]

  We have
  \begin{align*}
    n_k & \leq \rho^{k-1} n_1 = \frac{1}{\rho - 1} \bigl(\rho^{k}
          n_1 - \rho^{k-1} n_1 \bigr) \\
        & \leq \frac{1}{\rho - 1} (n_{k +1} - n_k + 1)
          \leq \frac{2}{\rho - 1} (n_{k+1} - n_k).
  \end{align*}
  It now follows that
  \begin{align*}
    \frac{1}{a_{n_k}} \sum_{l = n_k + 1}^{n_{k+1}} x_l \leq
    \sum_{l = n_k + 1}^{n_{k+1}} \frac{x_l}{a_l} = \Delta_k \leq
    (n_{k+1} - n_k) \Bigl(1 + \varepsilon + \frac{4 \varepsilon}{\rho
      - 1} \Bigr)
  \end{align*}
  and
  \[
    \frac{1}{a_{n_{k+1}}} \sum_{l = n_k + 1}^{n_{k+1}} x_l \geq
    (n_{k+1} - n_k) \Bigl(1 - \varepsilon - \frac{4
      \varepsilon}{\rho - 1} \Bigr).
  \]

  Letting $c_\rho$ be so that
  \[
    1 \leq \frac{a_n}{a_{[\rho n]}} \leq c_\rho
  \]
  for all $n \geq n_1$, we have
  \[
    \sum_{l = n_k + 1}^{n_{k+1}} x_l \leq \sum_{l = n_k +
      1}^{n_{k+1}} a_l c_\rho \Bigl(1 + \varepsilon + \frac{4
      \varepsilon}{\rho - 1} \Bigr)
  \]
  and
  \[
    \sum_{l = n_k + 1}^{n_{k+1}} x_l \geq \sum_{l = n_k +
      1}^{n_{k+1}} a_l c_\rho \Bigl(1 - \varepsilon - \frac{4
      \varepsilon}{\rho - 1} \Bigr).
  \]
  Hence
  \[
    c_\rho \Bigl(1 - \varepsilon - \frac{4 \varepsilon}{\rho - 1}
    \Bigr) \leq \frac{\sum_{l = n_1 + 1}^{n_m} x_l}{\sum_{l = n_1
        + 1}^{n_m} a_l} \leq c_\rho \Bigl(1 + \varepsilon +
    \frac{4 \varepsilon}{\rho - 1} \Bigr)
  \]
  holds for any $m > 1$.

  As will be proved below, from the assumption
  $\sum_{k = 1}^\infty a_k = \infty$ it now follows that
  \begin{equation}
    \label{eq:limsup}
    \limsup_{n \to \infty}
    \frac{\sum_{l = 1}^n x_l}{\sum_{l = 1}^n a_l} \leq \rho
    c_\rho \Bigl( 1 + \varepsilon + \frac{4 \varepsilon}{\rho -
      1} \Bigr)
  \end{equation}
  and
  \begin{equation}
    \label{eq:liminf}
    \liminf_{n \to \infty}
    \frac{\sum_{l = 1}^n x_l}{\sum_{l = 1}^n a_l} \geq
    \frac{c_\rho}{\rho} \Bigl( 1 - \varepsilon - \frac{4
      \varepsilon}{\rho - 1} \Bigr).
  \end{equation}
  This is proved as follows. For $n$ between $n_m$ and $n_{m+1}$
  we have
  \[
    \frac{\sum_{l = 1}^n x_l}{\sum_{l = 1}^n a_l} \leq
    \frac{\sum_{l = 1}^{n_{m+1}} x_l}{\sum_{l = 1}^{n_{m+1}} a_l}
    \frac{\sum_{l = 1}^{n_{m+1}} a_l}{\sum_{l = 1}^{n_m} a_l}
  \]
  and
  \begin{align*}
    \frac{\sum_{l = 1}^{n_{m+1}} a_l}{\sum_{l = 1}^{n_m} a_l}
    &= 1 + \frac{\sum_{l = n_m + 1}^{n_{m+1}} a_l}{\sum_{l = 1}^{n_m}
      a_l} \\
    &\leq 1 + \frac{(n_{m+1} - n_m) a_{n_m}}{n_m a_{n_m}} =
      1 + \frac{n_{m+1} - n_m}{n_m} \to \rho
  \end{align*}
  as $m \to \infty$. This proves \eqref{eq:limsup}, and the proof
  of \eqref{eq:liminf} is obtained in a similar way.

  We first note that we can let $c_\rho$ be fixed while
  increasing $n_1$. Letting $\varepsilon \to 0$, we may therefore
  keep $\rho$ and $c_\rho$ fixed. This proves that
  \begin{equation*}
    \limsup_{n \to \infty} \frac{\sum_{l = 1}^n x_l}{\sum_{l =
        1}^n a_l} \leq \rho c_\rho
  \end{equation*}
  and
  \begin{equation*}
    \liminf_{n \to \infty} \frac{\sum_{l = 1}^n x_l}{\sum_{l =
        1}^n a_l} \geq \frac{c_\rho}{\rho}.
  \end{equation*}
  Because of \eqref{eq:akassumption}, we can make $\rho$ and
  $c_\rho$ as close to $1$ as we desire, and hence
  \[
    \lim_{n \to \infty} \frac{\sum_{l = 1}^n
      x_l}{\sum_{l = 1}^n a_l} = 1
  \]
  which finishes the proof.
\end{proof}

\section{Proof of the Theorem}
\label{sec:proof}

We start the proof of the Theorem by establishing the following
estimate.

\begin{proposition} \label{prop:estimates}
  Suppose that $\mu$ is a probability measure and that $(E_n)$ is
  a sequence of sets that satisfy the estimate
  \begin{equation}
    \label{eq:Eest1}
    \mu (E_k \cap E_{k+l}) \leq \mu (E_k) \mu (E_{k+l}) + c
    e^{-\eta k} + c e^{-\eta l},
  \end{equation}
  for some $c, \eta > 0$. Then there is a constant $c_2$ such
  that for any $m < n$
  \[
    \int \biggl( \sum_{m < k \leq n} ( \mu(E_k)^{-1}
    \mathbf{1}_{E_k} - 1) \biggr)^2 \, \mathrm{d} \mu \leq c_2
    \sum_{m < k \leq n} \frac{1}{k^3 \mu (E_k)} \sum_{m < k \leq
      n} \frac{\log k}{\mu (E_k)}.
  \]
\end{proposition}

\begin{proof}
  We have of course that
  $\mu (E_k \cap E_{k+j}) \leq \mu (E_{k+j})$.

  Put
  \[
    f_k = \frac{1}{\mu (E_k)} \mathbf{1}_{E_k}
  \]
  and
  \[
    S = \int \biggl( \sum_{m < k \leq n} ( \mu(E_k)^{-1}
    \mathbf{1}_{E_k} - 1) \biggr)^2 \, \mathrm{d} \mu.
  \]
  Then
  \[
    \int f_k \, \mathrm{d} \mu = 1 \quad \text{and} \quad S =
    \sum_{m < j,k \leq n} \biggl( \int f_k f_j \, \mathrm{d} \mu
    - 1 \biggr).
  \]

  Let $c_0$ be a positive constant that will be chosen later.  We
  write $S$ as a sum of three parts, $S = A + B + C$, where
  \begin{align*}
    A &= 2 \sum_{m < k \leq n} \sum_{j = m + c_0 \log k}^{k - c_0
        \log k} \biggl( \int f_k f_j \, \mathrm{d} \mu - 1
        \biggr), \\
    B &= 2 \sum_{m < k \leq n} \sum_{j = k - c_0 \log k + 1}^k
        \biggl( \int f_k f_j \, \mathrm{d} \mu - 1 \biggr), \\
    C &= 2 \sum_{m < k \leq n} \sum_{j = m + 1}^{j = m + c_0 \log
        k} \biggl( \int f_k f_j \, \mathrm{d} \mu - 1 \biggr).
  \end{align*}
  We shall now estimate $A$, $B$ and $C$ separately.

  To estimate $A$, we use the assumption \eqref{eq:Eest1}, which
  implies that
  \[
    \mu (E_k \cap E_j) \leq \mu (E_k) \mu (E_j) + c e^{-\eta k} +
    c e^{- \eta (k-j)}
  \]
  for $j < k$. Hence
  \[
    \int f_k f_j \, \mathrm{d} \mu \leq 1 + \frac{1}{\mu (E_k)
      \mu (E_j)} (c e^{-\eta k} + c e^{- \eta (k-j)}).
  \]
  We may therefore estimate $A$ by
  \begin{align*}
    A &\leq 2 \sum_{m < k \leq n} \frac{1}{\mu (E_k)} \sum_{j = m
        + c_0 \log k}^{k - c_0 \log k} \frac{(c e^{-\eta k} + c
        e^{- \eta (k-j)})}{\mu (E_j)} \\
      & \leq 2 \sum_{m < k \leq
        n} \frac{c e^{- \eta k} + c e^{- \eta c_0 \log k}}{\mu (E_k)}
        \sum_{m < j \leq k} \frac{1}{\mu (E_j)}.
  \end{align*}
  Choose $c_0$ so large that $e^{- \eta c_0 \log k} \leq
  k^{-3}$. Then
  \[
    A \leq c_1 \biggl( \sum_{m < k \leq n} \frac{1}{\mu (E_k)}
    \biggr) \biggl( \sum_{m < k \leq n} \frac{1}{k^3 \mu (E_k)}
    \biggr),
  \]
  for some constant $c_1$.
  
  The terms $B$ and $C$ are both estimated in the following
  way. In the sums defining $B$ and $C$, there are not more that
  $c_0 \log k$ terms in the sum over $j$. We use the trivial
  estimate
  \[
    \int f_k f_j \, \mathrm{d} \mu \leq \frac{1}{\mu (E_k)} \int
    f_j \, \mathrm{d} \mu = \frac{1}{\mu (E_k)}.
  \]
  and obtain that
  \[
    B, C \leq 2 \sum_{m < k \leq n} \frac{c_0 \log k}{\mu (E_k)}.
  \]

  Combining all estimates, we get that
  \[
    S \leq c_2 \sum_{m < k \leq n} \frac{1}{k^3 \mu (E_k)}
    \sum_{m < k \leq n} \frac{\log k}{\mu (E_k)},
  \]
  for some constant $c_2$.
\end{proof}

We are now ready to prove the Theorem.

Let
\[
  E_k = \{\, x : d(T^k (x), x) < r_k (x) \,\} = \{\, x : T^k (x)
  \in B_k (x) \,\}
\]
with $r_k (x)$ defined by the relation
\[
  \mu (B (x, r_k(x))) = m_k.
\]
Then certain estimates hold. More precisely, under the
assumptions of the Theorem, we have that
\begin{align*}
  | \mu (E_k) - m_k | & \leq c e^{-\eta k}, \\
  \mu (E_k \cap E_{k+j}) & \leq \mu (E_k) \mu (E_{k+j}) + c e^{-\eta
                           k} + c e^{-\eta j}
\end{align*}
for some constants $c, \eta > 0$ \cite[Lemma~4.1 and
4.2]{KirsebomKundePersson}.

We put $f_k = \mu (E_k)^{-1} \mathbf{1}_{E_k}$. By
Proposition~\ref{prop:estimates}, we have
\[
  \int \biggl( \sum_{m < k \leq n} ( \mu(E_k)^{-1}
  \mathbf{1}_{E_k} - 1) \biggr)^2 \, \mathrm{d} \mu \leq c_2
  \sum_{m < k \leq n} \frac{1}{k^3 \mu (E_k)} \sum_{m < k \leq n}
  \frac{\log k}{\mu (E_k)}.
\]
Since $\mu(E_k)^{-1} \leq k (\log k)^{- 4 - \varepsilon}$, we have
\begin{align*}
  \int \biggl( \sum_{m < k \leq n} ( \mu(E_k)^{-1}
  \mathbf{1}_{E_k} - 1) \biggr)^2 \, \mathrm{d} \mu 
  & \leq c_4 \sum_{m < k \leq n} k (\log k)^{-3 - \varepsilon}.
\end{align*}

From Lemma~\ref{lem:Sprindzuk} follows that for almost all $x$
\begin{align*}
  \sum_{k = 1}^n \frac{1}{\mu (E_k)} \mathbf{1}_{B_k (x)}
  (T^k (x))
  &= \sum_{k = 1}^n \frac{1}{\mu(E_k)} \mathbf{1}_{E_k}
    (x) \\
  & = n + O \Bigl( \Phi (n + 1)^\frac{1}{2} (\log \Phi
    (n + 1))^{\frac{3}{2} + \frac{\varepsilon}{5}} \Bigr),
\end{align*}
where
\[
  \Phi (n) = \sum_{k = 1}^n c_4 k (\log k)^{-3 - \varepsilon}
  \leq c_5 n^2 (\log n)^{-3 - \varepsilon}.
\]
Hence
\[
  \sum_{k = 1}^n \frac{1}{\mu (E_k)} \mathbf{1}_{B_k (x)} (T^k
  (x)) = n + O \Bigl( n (\log
  n)^{- \frac{\varepsilon}{4}} \Bigr).
\]
In particular, since $\mu (E_k) / \mu (B_k (x)) \to 1$, we have
\[
  \lim_{n \to \infty} \frac{1}{n} \sum_{k = 1}^n \frac{1}{\mu
    (B_k (x))} \mathbf{1}_{B_k (x)} (T^k (x)) = 1
\]
for almost all $x$.  Finally, Lemma~\ref{lem:abel} implies that
\[
  \lim_{n \to \infty} \frac{\sum_{k = 1}^n \mathbf{1}_{B_k (x)}
    (T^k (x))}{\sum_{k = 1}^n \mu (B_k (x))} = 1,
\]
for almost every $x$.

\section{Proof of the Corollary} \label{sec:corollary}

In this section, we assume that $\mu$ is a non-atomic probability
measure, not necessarily invariant under $T$. Under this
assumption we prove two propositions, from which the Corollary
of this paper immediately follows.

\begin{proposition} \label{pro:liminf}
  Suppose that correlations are summable for $L^1$ against $BV$.
  For almost every $x$ we have that
  \[
    \liminf_{r \to 0} \frac{\log \tau_{B(x, r)} (x)}{- \log \mu
      (B(x, r))} \geq 1.
  \]
\end{proposition}

\begin{proof}
  Suppose that $x$ is such that
  \[
    \liminf_{r \to 0} \frac{\log \tau_{B(x, r)} (x)}{- \log \mu
      (B(x, r))} < \theta < 1.
  \]
  We may then choose sequences $r_n \to 0$ and $M_n \to \infty$
  such that
  \[
    \frac{\log \tau_{B(x, r_n)} (x)}{- \log \mu (B(x, r_n))} <
    \theta \qquad \text{and} \qquad 2^{- M_n - 1} \leq \mu (B(x,
    r_n)) \leq 2^{- M_n},
  \]
  for all $n$.

  Define the function $\rho_k$ such that
  \[
    \mu (B (y, \rho_k (y))) = 2 k^{-\frac{1}{\theta}}
  \]
  for all $y$. Put $k_n = 2^{(M_n + 1) \theta}$. Then, for the
  point $x$, we have
  \[
    B (x, r_n) \subset B (x, \rho_{k_n})
  \]
  since
  $\mu (B (x, \rho_{k_n})) = 2^{- M_n} \geq \mu (B (x,
  r_n))$. We then have
  \[
    \tau_{B(x, \rho_{k_n} (x))} \leq \tau_{B(x, r_n)} (x) \leq
    \mu (B (x, r_n))^{- \theta} \leq 2^{(M_n + 1)\theta} = k_n.
  \]
  It follows that there are infinitely many $k$ such that
  $T^k (x) \in B (x, \rho_k (x))$. But the set of $x$ with this
  property is a set of zero measure since $k^{-\frac{1}{\theta}}$
  is a summable sequence \cite[Theorem~C]{KirsebomKundePersson}.
\end{proof}

It now remains to prove the following proposition, for which no
mixing assumption is needed.

\begin{proposition}
  Suppose that $x$ is a point such that
  \[
    \lim_{n \to \infty} \frac{\sum_{k=1}^n \mathbf{1}_{B (x,
        r_k)} (T^k x)}{\sum_{k = 1}^n \mu (B (x,r_k))} = 1,
  \]
  where $r_k$ is such that
  $\mu (B(x, r_k)) = \frac{(\log k)^5}{k}$. Then
  \[
    \limsup_{r \to 0} \frac{\log \tau_{B(x, r)} (x)}{- \log \mu
      (B(x, r))} \leq 1.
  \]
\end{proposition}

\begin{proof}
  To get a contradiction, suppose that $x$ is such that
  \[
    \limsup_{r \to 0} \frac{\log \tau_{B(x, r)} (x)}{- \log \mu
      (B(x, r))} > \theta > 1.    
  \]
  Put $m_n = \frac{(\log n)^5}{n}$. There is a sequence $\rho_n$
  such that
  \[
    m_{n+1} \leq \mu (B (x, \rho_n)) \leq m_n
  \]
  and a sequence $n_j$ such that
  \[
    \tau_{B(x, \rho_{n_j})} (x) \geq \mu (B (x,
    \rho_{n_j}))^{-\theta}.
  \]
  It then follows that
  \[
    \tau_{B(x, \rho_{n_j})} (x) \geq \frac{n_j^\theta}{(\log
      n_j)^{5\theta}}.
  \]
  This means that if we put
  \[
    S_n (x) = \sum_{k=1}^n \mathbf{1}_{B (x, \rho_k)} (T^k x),
  \]
  then $S_{n_j} (x) = S_{N_j} (x)$, where
  $N_j = \frac{n_j^\theta}{(\log n_j)^{5\theta}}$.

  However, we have
  \begin{equation}
    \label{eq:nj-Nj}
    \frac{\sum_{k = 1}^{n_j} \mu (B (x, \rho_k))}{\sum_{k =
        1}^{N_j} \mu (B (x, \rho_k))} \leq \frac{ \sum_{k =
        1}^{n_j} m_k}{ \sum_{k = 1}^{n_j} m_k + \sum_{k = n_j +
        1}^{N_j} m_{k+1}} \leq \frac{1}{1 + K}
  \end{equation}
  for some constant $K > 0$, and
  \begin{equation}
    \label{eq:r-rho}
    \lim_{n \to \infty} \frac{\sum_{k = 1}^{n} \mu (B(x,
      r_k))}{\sum_{k = 1}^{n} \mu (B(x, \rho_k))} = 1.
  \end{equation}

  Now, because $S_{n_j} (x) = S_{N_j} (x)$ we have
  \begin{multline*}
    \frac{S_{n_j} (x)}{\sum_{k = 1}^{n_j} \mu(B(x,r_k))} \\ =
    \frac{S_{N_j} (x)}{\sum_{k = 1}^{N_j} \mu(B(x,\rho_k))}
    \frac{\sum_{k = 1}^{N_j} \mu(B(x,\rho_k))}{\sum_{k = 1}^{n_j}
      \mu(B(x,\rho_k))} \frac{\sum_{k = 1}^{n_j}
      \mu(B(x,\rho_k))} {\sum_{k = 1}^{n_j} \mu(B(x,r_k))},
  \end{multline*}
  and it follows from \eqref{eq:nj-Nj} and \eqref{eq:r-rho} that
  we cannot have
  \[
    \lim_{n \to \infty} \frac{S_n (x)}{\sum_{k = 1}^n \mu (B
      (x,r_k))} = 1,
  \]
  which is a contradiction.
\end{proof}

\section{Slowly mixing systems} \label{sec:slowmixing}

For the types of systems considered in the Theorem, it is known
\cite[Theorem~C]{KirsebomKundePersson} that if
$\sum m_k < \infty$, then for almost every $x$ we have that
$T^k (x) \in B(x, r_k (x))$ for at most finitely many $k$. For
some systems, it is furthermore known that if
$\sum m_k = \infty$, then for almost all $x$ we have that
$T^k (x) \in B(x, r_k (x))$ for infinitely many $k$
\cite{Husseinetal}, \cite[Theorem~D]{KirsebomKundePersson}.

It would not be unreasonable to suspect that for some systems we
have the following stronger dichotomy: The conclusion of the
Theorem holds whenever $(m_k)$ is a sequence of non-negative
numbers that satisfies $\sum m_k = \infty$, and if
$\sum m_k < \infty$, then for almost every $x$ we have that
$T^k (x) \in B(x, r_k (x))$ for at most finitely many $k$.

Because of technical reasons in the proof, we have not been able
to prove this dichotomy. However, it is possible to provide
examples of slowly mixing systems for which this dichotomy does
not hold. The purpose of this section is to give two such
examples.

\begin{example}[Rotations]
  Rotations are not mixing and one may therefore regard them as
  particularly slowly mixing. Let $\alpha$ be a real number and
  let $T_\alpha \colon \mathbf{S}^1 \to \mathbf{S}^1$ be the
  rotation by the angle $\alpha$ on the one dimensional circle
  $\mathbf{S}^1$. The Lebesgue measure is invariant, but not
  mixing.

  It is clear that for any $x \in \mathbf{S}^1$, we have
  $T_\alpha^k (x) \in B(x,r_k)$ if and only if
  $T_\alpha^k (0) \in B(0,r_k)$.

  If $\alpha$ is a badly approximable number, then there is a
  $c > 0$ such that
  \[
    d(x, T_\alpha^k (x)) = d(0, T_\alpha^k (0)) > m_k :=
    \frac{c}{k}
  \]
  for all $k > 0$. In this case we have $\sum m_k =
  \infty$, but for every $x$, there is no $k >
  0$ such that $T_\alpha^k (x) \in B(x, r_k)$.
  
  If $\alpha$ is of a higher Diophantine class, say for instance
  such that $d(0, T_\alpha^k (0)) < \frac{1}{k^2}$ for infinitely
  many $k$ (as is also the case if $\alpha$ is rational), then
  for these $k$ we have
  \[
    d(x, T_\alpha^k (x)) = d(0, T_\alpha^k (0)) < m_k :=
    \frac{1}{k^2},
  \]
  so that $\sum m_k < \infty$ but for every $x$ we have
  $T_\alpha^k (x) \in B(x, r_k)$ for infinitely many $k$.

  In fact \cite{KoukoulopoulosMaynard}, for almost all $\alpha$,
  we have for all $x$ that $T_\alpha^k(x) \in B(x,r_k)$ for
  infinitely many $k$, if and only if
  $\sum \phi (k) r_k = \infty$, where $\phi$ is Euler's totient
  function.
\end{example}

\begin{example}[Certain skew products]
  Here we consider a simple special case of the type of systems
  considered by Galatolo, Rousseau and Saussol
  \cite{GalatoloRousseauSaussol}. As in the previous example,
  $T_\alpha$ is the rotation by the angle $\alpha$. We consider
  the skew product transformation
  $T \colon [0,1) \times \mathbf{S}^1 \to [0,1) \times
  \mathbf{S}^1$ defined by
  \[
    T (x,y) = (2 x \bmod 1, y + \alpha \mathbf{1}_{[0,
      \frac{1}{2})} (x)).
  \]
  The two dimensional Lebesgue measure, which we denote by $\mu$,
  is invariant under $T$, and $\mu$ is mixing, with a rate which
  depends on $\alpha$.
  
  Let $\gamma (\alpha)$ be the infimum of those $\beta$ such that
  there is a $c > 0$ with
  $d(0,T_\alpha^k (0)) > \frac{c}{k^\beta}$ for all $k >
  0$. Then, as a special case of a theorem by Galatolo, Rousseau
  and Saussol \cite[Theorem~23]{GalatoloRousseauSaussol}, we have
  for almost all $p = (x,y)$ that the number
  \[
    \underline{R} (p) = \liminf_{r \to 0} \frac{\log
      \tau_{B(p,r)} (p)}{- \log r}
  \]
  satisfies
  $\underline{R} (p) \leq 1 + \frac{2}{\gamma
    (\alpha)}$. Moreover, $T$ is mixing for Lipschitz continuous
  functions with a rate $O(n^{-\frac{1}{2 \gamma (\alpha)}})$
  \cite[Proposition~9]{GalatoloRousseauSaussol}.

  We may choose $\alpha$ such that $\gamma (\alpha) > 2$ and
  then $\underline{R} (p) \leq 1 + \frac{2}{\gamma(\alpha)} < 2$
  holds for almost all $p$. Since $\mu$ is the two dimensional
  Lebesgue measure, we have
  \[
    \liminf_{r \to 0} \frac{\log \tau_{B(p, r)} (p)}{- \log \mu
      (B(p, r))} = \frac{1}{2} \underline{R} (p) \leq \frac{1}{2}
    + \frac{1}{\gamma} < 1
  \]
  for almost every $p$. Take
  \[
    \theta \in \biggl(\frac{1}{2} + \frac{1}{\gamma}, 1 \biggr).
  \]
  Following the beginning of the proof of
  Proposition~\ref{pro:liminf}, we obtain that for almost all $p$
  we have $T^k (p) \in B(p, \rho_k)$ for infinitely many $k$, and
  that $m_k := \mu (B(p, \rho_k)) = 2 k^{-\frac{1}{\theta}}$.

  Hence, $\sum m_k < \infty$ but nevertheless we have for almost
  all $p$ that $T^k (p) \in B(p, \rho_k)$ for infinitely many
  $k$.
\end{example}

\end{document}